\documentclass[12pt,reqno]{amsart}
\usepackage{indentfirst, amssymb, amsmath, amsthm, mathrsfs, setspace, indentfirst, enumerate,  mathrsfs, amsmath, amsthm}
\usepackage[bookmarksnumbered, colorlinks, plainpages]{hyperref}
\usepackage{mathrsfs}
\usepackage{cite}

\newtheorem*{theoA}{Theorem A}
\newtheorem*{theoB}{Theorem B}
\newtheorem*{theoC}{Theorem C}
\newtheorem*{theoD}{Theorem D}
\newtheorem*{theoE}{Theorem E}
\newtheorem*{theoF}{Theorem F}

\newtheorem*{cor A}{Corollary A}
\newtheorem*{cor B}{Corollary B}

\newtheorem{theo}{Theorem}[section]
\newtheorem{lem}{Lemma}[section]

\newtheorem{prob}{Problem}[section]

\newcommand{\ol}{\overline}
\newcommand{\be}{\begin{equation}}
\newcommand{\ee}{\end{equation}}
\newcommand{\beas}{\begin{eqnarray*}}
\newcommand{\eeas}{\end{eqnarray*}}
\newcommand{\bea}{\begin{eqnarray}}
\newcommand{\eea}{\end{eqnarray}}

\numberwithin{equation}{section}
\begin{document}
\title[M\MakeLowercase{ultidimensional analogues of the Refined} B\MakeLowercase{ohr} \MakeLowercase{type inequalities}]{\LARGE M\MakeLowercase{ultidimensional analogues of the Refined} B\MakeLowercase{ohr} \MakeLowercase{type inequalities}}
\date{}
\author[M. B. A\MakeLowercase{hamed}, S. M\MakeLowercase{ajumder} \MakeLowercase{and} N. S\MakeLowercase{arkar}]{M\MakeLowercase{olla} B\MakeLowercase{asir} A\MakeLowercase{hamed}$^*$, S\MakeLowercase{ujoy} M\MakeLowercase{ajumder} \MakeLowercase{and} N\MakeLowercase{abadwip} S\MakeLowercase{arkar}}

\address{Molla Basir Ahamed,
	Department of Mathematics,
	Jadavpur University,
	Kolkata-700032, West Bengal, India.}
\email{mbahamed.math@jadavpuruniversity.in}

\address{Sujoy Majumder, Department of Mathematics, Raiganj University, Raiganj, West Bengal-733134, India.}
\email{sm05math@gmail.com}

\address{Nabadwip Sarkar, Department of Mathematics, Raiganj University, Raiganj, West Bengal-733134, India.}
\email{naba.iitbmath@gmail.com}

\renewcommand{\thefootnote}{}
\footnote{*\emph{Corresponding Author}: Molla Basir Ahamed.}

\renewcommand{\thefootnote}{\arabic{footnote}}
\setcounter{footnote}{0}

\begin{abstract}
The main aim of this article is to establish a sharp improvement of the classical Bohr inequality for bounded holomorphic mappings in the polydisk $\mathbb{D}^n$.We also prove two other sharp versions of the Bohr inequality in the setting of several complex variables: one by replacing the constant term with the absolute value of the function, and another by replacing it with the square of the absolute value of the function.Furthermore, we establish multidimensional analogues of known results concerning the modulus of the derivative of analytic functions in the unit disk $\mathbb{D}$, replacing the derivative with the radial derivative of holomorphic functions in $\mathbb{D}^n$.All of the established results are shown to be sharp.
\end{abstract}
\thanks{Typeset by \AmS -\LaTeX}
\maketitle

\tableofcontents
\section{\bf Introduction and Preliminaries}
The classical theorem of Harald Bohr \cite{Bohr-PLMS-1914}, originally examined a century ago, continues to generate intensive research on what is now known as Bohr's phenomenon. Renewed interest surged in the 1990s due to successful extensions to holomorphic functions of several complex variables and to more abstract functional analytic settings. For instance, in 1997, Boas and Khavinson \cite{Boas-Khavinson-PAMS-1997} introduced and determined the $n$-dimensional Bohr radius for the family of holomorphic functions bounded by unity on the polydisk (see Section \ref{Sub-Sec-1.3} for a detailed discussion). This seminal work stimulated significant research interest in Bohr-type questions across diverse mathematical domains.Subsequent investigations have yielded further results on Bohr's phenomenon for multidimensional power series. Notable contributions in this area include those by Aizenberg \cite{Aizen-PAMS-2000,Aizenberg,Aizenberg-SM-2007}, Aizenberg \textit{et al.} \cite{Aizenberg-PAMC-1999,Aizenberg-SM-2005,Aizenberg-Aytuna-Djakov-JMAA-2001}, Defant and Frerick \cite{Defant-Frerick-IJM-2006}, and Djakov and Ramanujan \cite{Djakov-Ramanujan-JA-2000}. A comprehensive overview of various aspects and generalizations of Bohr's inequality can be found in \cite{Lata-Singh-PAMS-2022,Liu-Pon-PAMS-2021,Ali-Abu-Muhanna-Ponnusamy,Alkhaleefah-Kayumov-Ponnusamy-PAMS-2019,Defant-Frerick-AM-2011,Hamada-IJM-2012,Paulsen-Popescu-Singh-PLMS-2002,Paulsen-Singh-PAMS-2004,Paulsen-Singh-BLMS-2006}, the monographs by Kresin and Maz'ya \cite{Kresin-1903}, and the references cited therein. In particular, \cite[Section 6.4]{Kresin-1903} on Bohr-type theorems highlights rich opportunities to extend several existing inequalities to holomorphic functions of several complex variables and, significantly, to solutions of PDEs.
\subsection{\bf Classical Bohr inequality and its recent implications}
Let $ \mathcal{B} $ denote the class of analytic functions in the unit disk $\mathbb{D}:=\{{\zeta}\in\mathbb{C} : |{ \zeta}|<1\} $ of the form $ f({ \zeta})=\sum_{k=0}^{\infty}a_k {\zeta}^k $ such that $ |f({ \zeta})|<1 $ in $ \mathbb{D} $.  In the study of Dirichlet series, in $ 1914$, Harald Bohr \cite{Bohr-PLMS-1914} discovered the following interesting phenomenon 

\begin{theoA} If $ f\in\mathcal{B}$, then the following sharp inequality holds:
\begin{align}
\label{b1}	M_f(r):=\sum_{k=0}^{\infty}|a_k|r^k\leq 1\;\;\mbox{for}\;\; |{ \zeta}|=r\leq \frac{1}{3}.
\end{align}
\end{theoA}
We remark that if $|f({ \zeta})|\leq 1$ in $\mathbb{D}$ and $|f({ \zeta}_0)|=1$ for some point ${ \zeta}_0$ in $\mathbb{D}$, then $f$ must be a unimodular constant; therefore, we restrict our attention to non-constant functions $f\in \mathcal{B}$. The inequality (\ref{b1}), together with the sharp constant $1/3$, is known as the \textit{classical Bohr inequality}, and $1/3$ is called the \textit{Bohr radius} for the family $\mathcal{B}$. Bohr \cite{Bohr-PLMS-1914} originally established the inequality only for $r\le 1/6$, and Wiener later proved that the constant $1/3$ is sharp. Several alternative proofs have since appeared (see \cite{Paulsen-Popescu-Singh-PLMS-2002}-\cite{Paulsen-Singh-BLMS-2006}, \cite{Sidon-MZ-1927}, \cite{Tomic-MS-1962}, and the survey chapters \cite{Ali-Abu-Muhanna-Ponnusamy} and \cite[Chapter 8]{Garcia-Mashreghi-Ross}). Several of these works employ methods from complex analysis, functional analysis, number theory, and probability, and they further develop new theory and applications of Bohr's ideas on Dirichlet series. For example, various multidimensional generalizations of this result have been obtained in \cite{Aizenberg-PAMC-1999,Aizenberg-Aytuna-Djakov-JMAA-2001,Aizenberg,Boas-Khavinson-PAMS-1997,Djakov-Ramanujan-JA-2000,Jia-Liu-Ponnusamy-AMP-2025}.
\vspace{1.2mm}

It is worth noting that if $|a_0|$ in Bohr's inequality is replaced by $|a_0|^2$, then the Bohr radius improves from 
$1/3$ to $1/2$. Moreover, if $a_0=0$ in Theorem A, the sharp Bohr radius can be further improved to $1/\sqrt{2}$ (see, for example, \cite{Kayumov-Ponnusamy-CMFT-2017}, \cite[Corollary 2.9]{Paulsen-Popescu-Singh-PLMS-2002}, and the recent work \cite{Ponnusamy-Wirths-CMFT-2020} for a general result). These improvements rely on sharp coefficient estimates such as
\begin{align*}
|a_n|\leq 1-|a_0|^2,\quad n\geq 1,\;f\in\mathcal{B}.
\end{align*}

Using these inequalities, in \cite{Kayumov-Ponnusamy-CMFT-2017} it is observed that the sharp form of Theorem A cannot be obtained in the extremal case $|a_0|<1$. Nevertheless, a sharp version of Theorem A has been achieved for each individual function in $\mathcal{B}$ (see \cite{Alkhaleefah-Kayumov-Ponnusamy-PAMS-2019}), and for several subclasses of univalent functions (see \cite{Aizenberg, Muhanna-CVEE-2010}).\vspace{1.2mm}

{ Let $\mathbb{N}$ denote the set of positive integers. There is also notion of Rogosinski radius \cite{Rogosinski-1923} which is as follows: If $|f({ \zeta})|<1$ on $\mathbb{D}$, then we have $|S_N(\zeta)|\leq 1$ for $|{ \zeta}|\leq 1/2$ for every $N\in\mathbb{N}$, where $S_N(\zeta)=\sum_{n=1}^{N-1}a_n{ \zeta}^n$. The constant $1/2$ is sharp and called the \emph{Rogosinski radius}.}\vspace{1.2mm}

 In \cite{LLP}, Liu \textit{et al.} established a Refined Bohr-Rogosinski inequalities by applying a strengthened system of coefficient inequalities (see \cite[Lemma 4]{LLP}).

\begin{theoB}\cite[Theorem 1]{LLP}
Suppose that $f\in \mathcal{B}$ and $f({{ \zeta}})=\sum_{n=0}^{\infty} a_n { \zeta}^n.$ For $n\ge N$, let $t=\left[\frac{N-1}{2}\right]$. Then
\begin{align*}
|f({ \zeta})|+\sum_{n=N}^{\infty} |a_n| r^n+\operatorname{sgn}(t)\sum_{n=1}^{t} |a_n|^{2}\frac{r^{N}}{1-r}+\left( \frac{1}{1+|a_{0}|}+\frac{r}{1-r}\right)\sum_{n=t+1}^{\infty} |a_n|^{2} r^{2n}\le 1
\end{align*}
for $|{ \zeta}|=r\le R_{N}$, where $R_{N}$ is  the positive root of the equation $2(1+r)r^N-(1-r)^2=0$.
The radius $R_{N}$ is best possible.
Moreover,
\begin{align*}
|f({ \zeta})|^{2}+\sum_{n=N}^{\infty} |a_n| r^n+\operatorname{sgn}(t)\sum_{n=1}^{t} |a_n|^{2}\frac{r^{N}}{1-r}+\left( \frac{1}{1+|a_{0}|}+\frac{r}{1-r}\right)\sum_{n=t+1}^{\infty} |a_n|^{2} r^{2n}\le 1
\end{align*}
for $|{ \zeta}|=r\le R_{N}'$, where $R_{N}'$ is is  the positive root of the equation $2(1+r)r^N-(1-r)^2=0$.  
The radius $R_{N}'$ is best possible.
\end{theoB}
For the case $N=1$, it is straightforward to verify that $R_1=\sqrt{5}-2$ and $R_1'=1/3$. However, both constants can be improved when one considers any individual function in $\mathcal{B}$ (in the setting of Theorem B with $N=1$). In this regard, Liu \textit{et al.} \cite{LLP} obtained the following result.

\begin{theoC} \cite[Theorem 2]{LLP}
Suppose that $f\in \mathcal{B}$ and $f({ \zeta})=\sum_{n=0}^{\infty} a_n { \zeta}^n.$ Then
\begin{align*}
A(\zeta, r):=|f({\zeta})|+\sum_{n=1}^{\infty} |a_n| r^n
+\left( \frac{1}{1+|a_{0}|}+\frac{r}{1-r}\right)
\sum_{n=1}^{\infty} |a_n|^{2} r^{2n}
\le 1
\end{align*}
for $|{ \zeta}|=r\le r_{a_{0}}=\dfrac{2}{\,3+|a_{0}|+\sqrt{5}(1+|a_{0}|)}$.
The radius $r_{a_{0}}$ is best possible and satisfies $r_{a_{0}}>\sqrt{5}-2$.
Moreover,
\begin{align*}
B({ \zeta}, r):=|f(\zeta)|^{2}+\sum_{n=1}^{\infty} |a_n| r^n
+\left( \frac{1}{1+|a_{0}|}+\frac{r}{1-r}\right)
\sum_{n=1}^{\infty} |a_n|^{2} r^{2n}
\le 1
\end{align*}
for $|{ \zeta}|=r\le r_{a_{0}}'$, where $r_{a_{0}}'$ is the unique positive root of
\begin{align*}
(1-|a_{0}|^{3}) r^{3} - (1+2|a_{0}|) r^{2} - 2r + 1 = 0.
\end{align*}
The radius $r_{a_{0}}'$ is best possible. Furthermore,
\begin{align*}
\frac{1}{3} < r_{a_{0}}' < \frac{1}{2+|a_{0}|}.
\end{align*}
\end{theoC}
Motivated by the work of Kayumov and Ponnusamy \cite{Kayumov-Ponnusamy-CRMASP-2018}, Bohr-type inequalities for the family $\mathcal{B}$ were investigated in \cite{Liu-Shang-Xu-JIA-2018}. This investigation considered several forms where the Taylor coefficients in the classical Bohr inequality are partially or entirely replaced by higher-order derivatives of $f$. We recall only two of these results here.
\begin{theoD} \cite[Theorem 2.1]{Liu-Shang-Xu-JIA-2018} Suppose that $f\in \mathcal{B}$ and $f({ \zeta})=\sum_{n=0}^{\infty} a_n \zeta^n.$ Then the following sharp inequality holds:
\begin{align*}
|f({ \zeta})|+|f'({ \zeta})|r+\sum_{n=2}^{\infty} |a_n|r^n\leq 1\quad \text{for}\;|{ \zeta}|=r\leq \frac{\sqrt{17}-3}{4}.
\end{align*}
\end{theoD}
\begin{theoE} \cite[Theorem 2.2, Corollary 2.3]{Liu-Shang-Xu-JIA-2018} Suppose that $N(\ge 2)$ is an integer, 
$f({ \zeta})=\sum_{k=0}^{\infty} a_k { \zeta}^k$ is analytic in $\mathbb{D}$ and $|f({ \zeta})|<1$ in $\mathbb{D}$. Then
\begin{align*}
|f({ \zeta})| + \sum_{k=N}^{\infty} \left| \frac{f^{(k)}({ \zeta})}{k!} \right| |{ \zeta}|^{k}\le 1
\quad \text{for } |{ \zeta}|=r \le R_{N},
\end{align*}
where $R_{N}$ is the minimum positive root of the equation
\begin{align*}
\psi_{N}(r)= (1+r)(1-2r)(1-r)^{\,N-1} - 2r^{N}= 0.
\end{align*}
The radius $R_{N}$ is the best possible. Moreover,
\begin{align*}
|f({ \zeta})|^{2}+ \sum_{k=N}^{\infty} \left| \frac{f^{(k)}({ \zeta})}{k!} \right| |z|^{k}\le 1
\quad \text{for } |{ \zeta}| = r \le R_{N}^{*},
\end{align*}
where $R_{N}^{*}$ is the positive root of the equation
\begin{align*}
(1+r)(1-2r)(1-r)^{\,N-1} - r^{N} = 0.
\end{align*}
The radius $R_{N}^{*}$ is the best possible.
\end{theoE}
Liu \emph{et al.} \cite{LLP} presented an improved version of Theorem D as follows.
\begin{theoF}\cite[Theorem 7]{LLP} Suppose that $f\in \mathcal{B}$ and $f({ \zeta})=\sum_{n=0}^{\infty} a_n \zeta^n.$
Then
\begin{align*}
I(\zeta, r):=|f({ \zeta})| + |f'({ \zeta})|\, r + \sum_{n=2}^{\infty} |a_n| r^n+ \left( \frac{1}{1+|a_{0}|} + \frac{r}{1-r} \right)\sum_{n=1}^{\infty} |a_n|^{2} r^{2n}\le 1
\end{align*}
for $|{\zeta}|=r \le \frac{\sqrt{17}-3}{4}$, and the constant $\frac{\sqrt{17}-3}{4}$ is best possible.
Moreover,
\begin{align*}
J({ \zeta}, r):=|f({ \zeta})|^{2} + |f'({ \zeta})|\, r + \sum_{n=2}^{\infty} |a_n| r^n+ \left( \frac{1}{1+|a_{0}|} + \frac{r}{1-r} \right)\sum_{n=1}^{\infty} |a_n|^{2} r^{2n}\le 1
\end{align*}
for $|{\zeta}|=r \le r_{0}$, where $r_{0}\approx 0.385795$ is the unique positive root of the equation
\begin{align*}
1 - 2r - r^{2} - r^{3} - r^{4} = 0,
\end{align*}
and $r_{0}$ is best possible.
\end{theoF}
{ For other results on Bohr-Rogosinski radius in one complex variable, see \textit{e.g.} \cite{Ahamed-RM-2023,Ahamed-Allu-BMMSS-2022,Ahamed-Allu-CMB-2023,Allu-Arora-JMAA-2023,Das-JMAA-2022,Gangania-Kumar-MJM-2022,Hamada-AAMP-2025}.\vspace{1.2mm}
	
	It would be interesting to study the Bohr-Rogosinski radius for functions with values in the polydisk in several complex variables. It is natural to raise the following question:
	\begin{prob}
		Can we establish the multidimensional versions of Theorems B, C, E, and F?
	\end{prob}
	In this paper, our aim is to present an affirmative answer to this problem. The organization of the paper is as follows: First, we discuss about some basic Notations and terminology in several complex variables. In Section \ref{Sec-2}, we present our main results. In Section \ref{Sec-3}, we present some key lemmas which will play a key role in proving the results of this paper. Finally, in Section \ref{Sec-4}, we present the proofs of our main results.}
\subsection{\bf Basic Notations in several complex variables}\label{Sub-Sec-1.3}
For $z=(z_1,\ldots,z_n)$ and $w=(w_1,\ldots,w_n)$ in $\mathbb{C}^{n}$, we denote $\langle z,w\rangle=z_1\ol w_1+\ldots+z_n \ol w_n$ and $||z||=\sqrt{\langle z,z\rangle}$. The absolute value of a complex number $z_1$ is denoted by $|z_1|$ and for $z\in\mathbb{C}^n$, we define
\begin{align*}
	||z||_{\infty}=\max\limits_{1\leq i\leq n}|z_i|.
\end{align*}

An open polydisk (or open polycylinder) in $\mathbb{C}^n$ is a subset $\mathbb{P}\Delta(a;r)\subset \mathbb{C}^n$ of the form 
\[\mathbb{P}\Delta(a;r)=\prod\limits_{j=1}^n \Delta(a_j;r_j)=\lbrace z\in\mathbb{C}^n: |z_i-a_i|<r_i,\;i=1,2,\ldots,n\rbrace,\]
the point $a=(a_1,\ldots,a_n)\in\mathbb{C}^n$ is called the centre of the polydisk and $r=(r_1,\ldots,r_n)\in\mathbb{R}^n\;(r_i>0)$ is called the polyradius. It is easy to see that
\begin{align*}
	\mathbb{P}\Delta(0;1)=\mathbb{P}\Delta(0;1_n)=\prod\limits_{j=1}^n \Delta(0_n;1_n).
\end{align*}

The closure of $\mathbb{P}\Delta(a;r)$ will be called the closed polydisk with centre $a$ and polyradius $r$ and will be denoted by $\ol{\mathbb{P}\Delta}(a;r)$. We denote by $C_k(a_k;r_k)$ the boundary of $\Delta(a_k;r_k)$, \textit{i.e.,} the circle of radius $r_k$ with centre $a_k$ on the $z_k$-plane. Of course $C_k(a_k,r_k)$ is represented by the usual parametrization 
\begin{align*}
	\theta_k\to \gamma(\theta_k)=a_k+r_ke^{i\theta_k},\; \mbox{where}, 0\leq \theta_k\leq 2\pi.
\end{align*} The product $C^n(a;r):=C_1(a_1;r_1)\times\ldots\times C_n(a_n;r_n)$ is called the determining set of the polydisk $\mathbb{P}\Delta(a;r)$.

\smallskip
A multi-index $\alpha=(\alpha_1,\ldots,\alpha_n)$ of dimension $n$ consists of n non-negative integers $\alpha_j,\;1\leq j\leq n$; the degree of a multi-index $\alpha$ is the sum $|\alpha|=\sum_{j=1}^n \alpha_j$ and we denote $\alpha!=\alpha_1!\ldots \alpha_n!$. For $z=(z_1,\ldots,z_n)\in\mathbb{C}^n$ and a multi-index $\alpha=(\alpha_1,\ldots,\alpha_n)$, we define 
\[z^{\alpha}=\prod\limits_{j=1}^n z_j^{\alpha_j}\;\;\text{and}\;\;|z|^{\alpha}=\prod\limits_{j=1}^n |z_j|^{\alpha_j}.\]

For two multi-indexes $\alpha=(\alpha_1,\ldots,\alpha_n)$ and $\nu=(\nu_1,\ldots,\nu_n)$, we define $\nu^{\alpha}=\nu_1^{\alpha_1}\ldots \nu_n^{\alpha_n}.$ 
Let $f(z)$ be a holomorphic function in a domain $\Omega\subset \mathbb{C}^n$, and $c=(c_1,\ldots,c_n)\in \Omega$. Then in a polydisk $\mathbb{P}\Delta(c;r)\subset \Omega$ with centre $c$, $f(z)$ has a power series expansion in $z_1-c_1,\ldots,z_n-c_n$,
\begin{align*}
	f(z)&=\sum\limits_{\alpha_1,\alpha_2,\ldots,\alpha_n=0}^{\infty} a_{\alpha_1,\alpha_2,\ldots,\alpha_n}(z_1-c_1)^{\alpha_1}(z_2-c_2)^{\alpha_2}\ldots (z_n-c_n)^{\alpha_n}\\&=
	\sum\limits_{|\alpha|=0} a_{\alpha}(z-c)^{\alpha}=\sum\limits_{|\alpha|=0}^{\infty} P_{|\alpha|}(z-c),
\end{align*}
which is absolutely convergent in $\mathbb{P}\Delta(c;r)$, where the term $P_k(z-c)$ is a homogeneous polynomial of degree $k$. \vspace{1.2mm}

One of our aims is to establish multidimensional analogues of Theorems D, E, and F. For this purpose, we consider $f$ to be a holomorphic function in $\Omega \subset \mathbb{C}^n$ and define the Euler operator $D$ as:$$Df(z) := \sum_{k=1}^n z_k\frac{\partial f(z)}{\partial z_k}.$$This operator is also known as the Euler derivative, the total derivative, or the radial derivative.
\section{{\bf Main Results}}\label{Sec-2}
We first state the multidimensional version of Theorem B.
\begin{theo}\label{Th-2.1} Let $f(z)=\sum_{|\alpha|=0} a_{\alpha}z^{\alpha}$ be a holomorphic function in the polydisk $\mathbb{P}\Delta(0;1_n)$ such that $|f(z)|\leq 1$ for all $z\in \mathbb{P}\Delta(0;1/n)$. Suppose $z=(z_1,\ldots,z_n)\in \mathbb{P}\Delta(0;1/n)$ and $r=(r_1,r_2,\ldots,r_n)$ such that $||z||_{\infty}={\bf r}$. For $n\geq N$, let $t=\left[\frac{N-1}{2} \right]$. Then
\begin{align*}
\mathcal{A}_1(z, \bf r):&=|f(z)|+\sum_{k=N}^{\infty}\sum\limits_{|\alpha|=k}|a_{\alpha}|\, r^{\alpha} + \operatorname{sgn}(t)\sum\limits_{k=1}^t\sum\limits_{|\alpha|=k}|a_{\alpha}|^2 \frac{{\bf r}^{N}}{1-{\bf r}}\\&\quad+\left( \frac{1}{1+|a_0|}+\frac{{\bf r}}{1-{\bf r}}\right)
\sum_{k=t+1}^{\infty} \sum\limits_{|\alpha|=k}|a_{\alpha}|^2 r^{2\alpha}\leq 1
\end{align*}
for $n{\bf r}\leq R_{n,N}$, where $R_{n,N}$ is the positive root of the equation $\psi_N(n{\bf r})=0$, 
$\psi_N(n{\bf r})=2(1 + n{\bf r}) (n{\bf r})^{N}-(1-n{\bf r})^{2}$. The radius $R_{n,N}$ is best possible. Moreover,
\begin{align*}
\mathcal{A}_2(z, \bf r):&=|f(z)|^2+\sum_{k=N}^{\infty}\sum\limits_{|\alpha|=k}|a_{\alpha}|\, r^{\alpha} + \operatorname{sgn}(t)\sum\limits_{k=1}^t\sum\limits_{|\alpha|=k}|a_{\alpha}|^2 \frac{{\bf r}^{N}}{1-{\bf r}}\\&\quad+\left( \frac{1}{1+|a_0|}+\frac{{\bf r}}{1-{\bf r}}\right)
\sum_{k=t+1}^{\infty} \sum\limits_{|\alpha|=k}|a_{\alpha}|^2 r^{2\alpha}\leq 1
\end{align*}
for $n{\bf r}\leq R'_{n,N}$, where $R'_{n,N}$ is the positive root of the equation $(1 + n{\bf r})(n{\bf r})^{N}-(1-n{\bf r})^{2} = 0$. The radius $R'_N$ is best possible.
\end{theo}
Next we state the multidimensional version of Theorem C.
\begin{theo}\label{Th-2.2} Let $f(z)=\sum_{|\alpha|=0} a_{\alpha}z^{\alpha}$ be a holomorphic function in the polydisk $\mathbb{P}\Delta(0;1_n)$ such that $|f(z)|\leq 1$ for all $z\in \mathbb{P}\Delta(0;1/n)$. Suppose $z=(z_1,\ldots,z_n)\in \mathbb{P}\Delta(0;1/n)$ and $r=(r_1,r_2,\ldots,r_n)$ such that $||z||_{\infty}={\bf r}$. Then
\begin{align*}
\mathcal{A}_3(z, {\bf r}):=|f(z)|+\sum_{k=1}^{\infty}\sum\limits_{|\alpha|=k}|a_{\alpha}|\, r^{\alpha}+\left( \frac{1}{1+|a_0|}+\frac{{\bf r}}{1-{\bf r}}\right)
\sum_{k=1}^{\infty} \sum\limits_{|\alpha|=k}|a_{\alpha}|^2 r^{2\alpha}\leq 1\nonumber
\end{align*}
for $n{\bf r}\leq nr_{a_0} = \frac{2}{\,3 + |a_0| + \sqrt{5}\,(1 + |a_0|)\,}$. 
The radius $nr_{a_0}$ is best possible and $nr_{a_0} > \sqrt{5} - 2$.  Moreover,
\begin{align*}
\mathcal{A}_4(z, {\bf r}):=|f(z)|^2+\sum_{k=1}^{\infty}\sum\limits_{|\alpha|=k}|a_{\alpha}|\, r^{\alpha}+\left( \frac{1}{1+|a_0|}+\frac{{\bf r}}{1-{\bf r}}\right)
\sum_{k=1}^{\infty} \sum\limits_{|\alpha|=k}|a_{\alpha}|^2 r^{2\alpha}\leq 1\nonumber
\end{align*}
for ${\bf r}\leq r'_{a_0}$, where $nr'_{a_0}$ is the unique positive root of the equation
\begin{align*}
(1 - |a_0|^3) (n{\bf r}^3 - (1 + 2|a_0|)\, (n{\bf r})^2 - 2n{\bf r} + 1 = 0.
\end{align*}
The radius $nr'_{a_0}$ is best possible. Further, we have $1/3n< r'_{a_0}<1/n(2 + |a_0|)$.
\end{theo}
Next we state the multidimensional version of Theorem E.
\begin{theo}\label{Th-2.3} Let $f(z)=\sum_{|\alpha|=0} a_{\alpha}z^{\alpha}$ be a holomorphic function in the polydisk $\mathbb{P}\Delta(0;1_n)$ such that $|f(z)|\leq 1$ for all $z\in \mathbb{P}\Delta(0;1/n)$. Suppose $z=(z_1,\ldots,z_n)\in \mathbb{P}\Delta(0;1/n)$ and $r=(r_1,r_2,\ldots,r_n)$ such that $||z||_{\infty}={\bf r}$. Then
\small{\begin{align*}
\mathcal{I}(z, {\bf r}):=|f(z)|+|Df(z)|+\sum_{k=2}^{\infty}\sum\limits_{|\alpha|=k}|a_{\alpha}|\, r^{\alpha}+\left( \frac{1}{1+|a_0|}+\frac{{\bf r}}{1-{\bf r}}\right)
\sum_{k=1}^{\infty} \sum\limits_{|\alpha|=k}|a_{\alpha}|^2 r^{2\alpha}\leq 1\nonumber
\end{align*}}
for $n{\bf r}\leq \frac{\sqrt{17}-3}{4}$ and the constant $\frac{\sqrt{17}-3}{4}$ is best possible. Moreover
\begin{align*}
\mathcal{J}(z, {\bf r}):=|f(z)|^2+|Df(z)|+\sum_{k=2}^{\infty}\sum\limits_{|\alpha|=k}|a_{\alpha}|\, r^{\alpha}+\left( \frac{1}{1+|a_0|}+\frac{{\bf r}}{1-{\bf r}}\right)
\sum_{k=1}^{\infty} \sum\limits_{|\alpha|=k}|a_{\alpha}|^2 r^{2\alpha}\leq 1\nonumber
\end{align*}
for ${\bf r}\leq r_{0}$, where $nr_0 \approx 0.385795$ is the unique positive root of the equation
\begin{align*}
1 - 2n{\bf r} - (n{\bf r})^{2} - (n{\bf r})^{3} - (n{\bf r})^{4} = 0
\end{align*}
and $nr_0$ is best possible.
\end{theo}
Finally, we state the multidimensional generalization of Theorem F.
\begin{theo}\label{Th-2.4} Let $f(z)=\sum_{|\alpha|=0} a_{\alpha}z^{\alpha}$ be a holomorphic function in the polydisk $\mathbb{P}\Delta(0;1_n)$ such that $|f(z)|\leq 1$ for all $z\in \mathbb{P}\Delta(0;1/n)$. Suppose $z=(z_1,\ldots,z_n)\in \mathbb{P}\Delta(0;1/n)$ and $r=(r_1,r_2,\ldots,r_n)$ such that $||z||_{\infty}={\bf r}$. Let $N\geq 2$. Then
\begin{align*}
\label{2.3} \mathcal{M}(z, {\bf r}):=|f(z)| +\sum_{k=N}^{\infty} \ \sum_{|\alpha|=k}\frac{1}{\alpha!}\left|\frac{\partial^{|\alpha|} f(z)}{\partial z_1^{\alpha_1}\cdots \partial z_n^{\alpha_n}}\right| |z|^{\alpha}\le 1
\end{align*}
for $n{\bf r} \le \tilde R_{n,N}$, where $\tilde R_{n,N}$ is the minimum positive root of the equation
\begin{align*}
(1+n{\bf r})(1-2n{\bf r})(1-n{\bf r})^{N-1} - 2(n{\bf r})^N = 0.
\end{align*}

The radius $\tilde R_{n,N}$ is the best possible. Moreover 
\begin{align*}\mathcal{N}(z, {\bf r}):=|f(z)|^2 +\sum_{k=N}^{\infty} \ \sum_{|\alpha|=k}\frac{1}{\alpha!}\left|\frac{\partial^{|\alpha|} f(z)}{\partial z_1^{\alpha_1}\cdots \partial z_n^{\alpha_n}}\right| |z|^{\alpha}\le 1
\end{align*}
for $n{\bf r} \le \tilde R_{n,N}'$, where $\tilde R_{n,N}'$ is the minimum positive root of the equation
\begin{align*}
(1+n{\bf r})(1-2n{\bf r})(1-n{\bf r})^{N-1} - (n{\bf r})^N = 0.
\end{align*}
\end{theo}
\section{{\bf Key lemmas and their Proofs}}\label{Sec-3}
In order to establish our main results, we need the following lemmas. The first of these is a special case of \cite[Theorem 2.2 ]{Chen-Hamada-Ponnusamy-Vijayakumar-JAM-2024}.
\begin{lem}\label{Lem1}Let $f$ be holomorphic in the polydisk $\mathbb{P}\Delta(0;1_n)$ such that $|f(z)|\le 1$ for all $z\in \mathbb{P}\Delta(0;1_n)$. Then for all $z\in \mathbb{P}\Delta(0;1_n)$, we have
\[|f(z)|\leq \frac{|f(0)|+||z||_{\infty}}{1+|f(0)|||z||_{\infty}}.\]
\end{lem}
The following lemma is contained in \cite[Corollary 1.3]{Liu-Chen-IJPAM-2012}.
\begin{lem}\label{Lem2} Let $f$ be holomorphic in the polydisk $\mathbb{P}\Delta(0;1_n)$ such that $|f(z)|< 1$ for all $z\in \mathbb{P}\Delta(0;1_n)$. Then for multi-index $\alpha=(\alpha_1,\ldots,\alpha_n)$, we have
\begin{align*}
\left|\frac{\partial^{|\alpha|} f(z)}{\partial z_1^{\alpha_1}\ldots \partial z_n^{\alpha_n}}\right|\leq \alpha!\frac{1-|f(z)|^2}{(1-||z||_{\infty}^2)^{|\alpha|}}(1+||z||_{\infty})^{|\alpha|-N}
\end{align*}
for all $z\in \mathbb{P}\Delta(0;1_n)$, where $N$ is the number of the indices $j$ such that $\alpha_j\neq 0$.
\end{lem}
\begin{lem}\label{Lem3}\cite{Carlson-AMAF-1940} Suppose that $f\in\mathbb{B}$ and $f(z)=\sum_{k=0}^{\infty} a_k z^k$. Then the following inequalities hold.
\begin{enumerate} 
\item[\emph{(a)}] $|a_{2k+1}|\leq 1-|a_0|^2-|a_1|^2-\ldots-|a_k|^2$,\;\;$k=0,1,\ldots$;
\item[\emph{(b)}] $|a_{2k}|\leq 1-|a_0|^2-|a_1|^2-\ldots-|a_{k-1}|^2-\frac{|a_k|^2}{1+|a_0|}$, \;\;$k=1,2,\ldots$. 
\end{enumerate}
Further, to have equality in $(a)$ it is necessary that $f(z)$ is  a rational function of the form
\begin{align*}
	f(z)=\frac{a_0+a_1z+\ldots+a_kz^k+\epsilon z^{2k+1}}{1+\left(\ol a_k z^k+\ldots+\ol a_0 z^{2k+1}\right)\epsilon},\;\;|\epsilon|=1,
\end{align*}
and to have equality in $(b)$ it is necessary that $f(z)$ is a rational function of the form
\[f(z)=\frac{a_0+a_1z+\ldots+\frac{a_k}{1+|a_0|}z^k+\epsilon z^{2k}}{1+\left(\frac{\ol a_k}{1+|a_0|} z^k+\ldots+\ol a_0 z^{2k}\right)\epsilon},\;\;|\epsilon|=1,\]
where $a_0\ol a_k^2\epsilon$ is non-positive real.
\end{lem}
To establish our results, we first prove the following result, which is a multidimensional generalization of Lemma \ref{Lem3}.
\begin{lem}\label{Lem4} Let $f$ be a holomorphic function in the polydisk $\mathbb{P}\Delta(0;1_n)$ such that $|f(z)|\leq 1$ and
$f(z)=\sum_{|\alpha|=0}^{\infty}a_{\alpha}z^{\alpha}$ for all $z\in \mathbb{P}\Delta(0;1_n)$. Then the following inequalities hold.
\begin{enumerate} 
\item[\emph{(i)}] For $k=0,1,2,\ldots$, we have 
\begin{align*}
\sum_{|\alpha|=2k+1} |a_{\alpha}|\leq n^{(2k+1)/2}\left(1-\sum_{i=0}^k\sum_{|\alpha|=i}|a_{\alpha}|^2\right).
\end{align*}
\item[\emph{(ii)}] For $k=1,2,\ldots$, we have
\begin{align*}
 \sum_{|\alpha|=2k} |a_{\alpha}|\leq n^{k}\left(1-\sum_{i=0}^{k-1}\sum_{|\alpha|=i}|a_{\alpha}|^2-\frac{\sum_{|\alpha|=k}|a_{\alpha}|^2}{1+|a_0|}\right).
\end{align*}
\end{enumerate}
\end{lem}
\begin{proof} Let $f(z)=a_0+P_{1}(z)+P_{2}(z)+\ldots$, where $P_k(z)=\displaystyle\sum_{|\alpha|=k}a_{\alpha}z^{\alpha}$
is a homogeneous polynomial of degree $k$. We define 
\begin{align}
	\label{mms1}g(t)=f\left(z t\right)=a_0+P_{1}\left(z\right)t+P_{2}\left(z\right)t^2+\ldots,
\end{align}
where $t\in\mathbb{C}$ such that $|t|\leq 1$.
For a fixed $z\in \mathbb{P}\Delta(0;1_n)$, we can say that $g(t)$ is analytic in $|t|\leq 1$. 
Since $|f(z)|\leq 1$, from \eqref{mms1}, we get $|g(t)|\leq 1$ in $|t|\leq 1$. Now by Lemma \ref{Lem5} $(a)$, we have
\begin{align*} 
|P_{2k+1}(z)|\leq 1-|a_0|^2-|P_1(z)|^2-\ldots-|P_k(z)|^2,
\end{align*}
\textit{i.e.,}
\begin{align}
\label{mms2} |a_0|^2+|P_1(z)|^2+\ldots+|P_k(z)|^2+|P_{2k+1}(z)|\leq 1,
\end{align}
where $k=0,1,2,\ldots$.
We set $z_j=r_je^{i\theta_j}\;(0\leq\theta_j\leq 2\pi)$, where $0<r_j<1$, $j=1,2,\ldots,n$ and consider the integration of $\left|P_{i}\left(r_1e^{i\theta_1},\ldots,r_ne^{i\theta_n}\right)\right|^2$. Noting that $\int_0^{2\pi} e^{i(\beta-\gamma)\theta}d\theta=2\pi \delta_{\beta,\gamma}$ (Kronecker's symbol), we have
\begin{align*} \sum_{|\alpha|=i} |a_{\alpha}|^2r^{2\alpha}&=\frac{1}{(2\pi)^n} \int_{0}^{2\pi} \cdots \int_{0}^{2\pi} \left|P_{i}\left(r_1e^{i\theta_1},\ldots,r_ne^{i\theta_n}\right)\right|^2 \, d\theta_1 \cdots d\theta_n\\&\leq
\max_{C^n(0,r)}\left|P_{i}\left(r_1e^{i\theta_1},\ldots,r_ne^{i\theta_n}\right)\right|^2,
\end{align*}
\textit{i.e.,}
\begin{align}
\label{mms3} \sum_{|\alpha|=i} |a_{\alpha}|^2r^{2\alpha}\leq \max_{C^n(0,r)}\left|P_{i}\left(r_1e^{i\theta_1},\ldots,r_ne^{i\theta_n}\right)\right|^2,
\end{align}
where $i\geq 1$ and $r=(r_1,r_2,\ldots,r_n)$. On the other hand, from (\ref{mms2}), we have 
\begin{align}
\label{mms4} &|a_0|^2+\max_{C^n(0,r)}\left|P_1\left(r_1e^{i\theta_1},\ldots,r_ne^{i\theta_n}\right)\right|^2+\ldots+\max_{C^n(0,r)}\left|P_k\left(r_1e^{i\theta_1},\ldots,r_ne^{i\theta_n}\right)\right|^2\\&+\max_{C^n(0,r)}\left|P_{2k+1}\left(r_1e^{i\theta_1},\ldots,r_ne^{i\theta_n}\right)\right|\leq 1,\nonumber
\end{align}
where $k=0,1,2,\ldots$\vspace{1.2mm}

 Using (\ref{mms3}) to (\ref{mms4}), we obtain
\begin{align}\label{mms5} |a_0|^2+\sum_{i=1}^k \sum_{|\alpha|=i} |a_{\alpha}|^2r^{2\alpha}+\max_{C^n(0,r)}\left|P_{2k+1}\left(r_1e^{i\theta_1},\ldots,r_ne^{i\theta_n}\right)\right|\leq 1,
\end{align}
where $k=0,1,2,\ldots$.\vspace{1.2mm}

 By Cauchy-Schwarz inequality, we have
\begin{align*}
\left(\sum_{|\alpha|=2k+1} |a_{\alpha}r^{\alpha}| \right)^2\leq \left(\sum_{|\alpha|=2k+1} |a_{\alpha}|^2\right)\times \left(\sum_{|\alpha|=2k+1} r^{2\alpha}\right)\leq n^{|\alpha|}\;\sum_{|\alpha|=2k+1} |a_{\alpha}|^2r^{2\alpha}
\end{align*}
and thus it follows from (\ref{mms3}) that
\begin{align*} 
\left(\sum_{|\alpha|=2k+1} |a_{\alpha}r^{\alpha}| \right)^2\leq n^{|\alpha|}\left(\max_{C^n(0,r)}\left|P_{2k+1}\left(r_1e^{i\theta_1},\ldots,r_ne^{i\theta_n}\right)\right|\right)^2,
\end{align*}
\textit{i.e.,}
\begin{align}
\label{mms6} \sum_{|\alpha|=2k+1} |a_{\alpha}r^{\alpha}|\leq n^{|\alpha|/2}\max_{C^n(0,r)}\left|P_{2k+1}\left(r_1e^{i\theta_1},\ldots,r_ne^{i\theta_n}\right)\right|.
\end{align}
Therefore, using (\ref{mms6}) to (\ref{mms5}), we have
\begin{align}
\label{mms7} n^{|\alpha|/2}|a_0|^2+n^{|\alpha|/2}\sum_{i=1}^k \sum_{|\alpha|=i} |a_{\alpha}|^2r^{2\alpha}+\sum_{|\alpha|=2k+1} |a_{\alpha}r^{\alpha}|\leq n^{|\alpha|/2},
\end{align}
where $k=0,1,2,\ldots$ Letting $r_i\to 1$ for $i=1,2,\ldots,n$, we get from (\ref{mms7}) that 
\begin{align*}
\sum_{|\alpha|=2k+1} |a_{\alpha}|\leq n^{(2k+1)/2}\left(1-\sum_{i=0}^k\sum_{|\alpha|=i}|a_{\alpha}|^2\right),
\end{align*}
where $k=0,1,2,\ldots$.\vspace{1.2mm}

By Lemma \ref{Lem5} $(b)$, we have
\begin{align*} 
|P_{2k}(z)|\leq 1-|a_0|^2-|P_1(z)|^2-\ldots-|P_{k-1}(z)|^2-\frac{|P_k(z)|^2}{1+|a_0|},
\end{align*}
\textit{i.e.,}
\begin{align*} 
|a_0|^2+|P_1(z)|^2+\ldots+|P_{k-1}(z)|^2+\frac{|P_k(z)|^2}{1+|a_0|}+|P_{2k}(z)|\leq 1,
\end{align*}
where $k=1,2,\ldots$.\vspace{1.2mm}

Based on the similar argument used before, we prove that
\begin{align*}
\sum_{|\alpha|=2k} |a_{\alpha}|\leq n^{k}\left(1-\sum_{i=0}^{k-1}\sum_{|\alpha|=i}|a_{\alpha}|^2-\frac{\sum_{|\alpha|=k}|a_{\alpha}|^2}{1+|a_0|}\right),
\end{align*}
where $k=1,2,\ldots$
\end{proof}
\begin{lem}\label{Lem5} Let $f$ be a holomorphic function in the polydisk $\mathbb{P}\Delta(0;1_n)$ such that $|f(z)|\leq 1$ and
$f(z)=\sum_{|\alpha|=0}^{\infty}a_{\alpha}z^{\alpha}$ for all $z\in \mathbb{P}\Delta(0;1_n)$. Then for any $N\in\mathbb{N}$, the following sharp inequality holds:
\begin{align}
\label{Lema5}\sum_{k=N}^{\infty}\sum\limits_{|\alpha|=k}|a_{\alpha}|\, r^{\alpha} &+ \operatorname{sgn}(t)\sum\limits_{k=1}^t\sum\limits_{|\alpha|=k}|a_{\alpha}|^2 \frac{{\bf r}^{N}}{1-{\bf r}}+\left( \frac{1}{1+|a_0|}+\frac{{\bf r}}{1-{\bf r}}\right)
\sum_{k=t+1}^{\infty} \sum\limits_{|\alpha|=k}|a_{\alpha}|^2 r^{2\alpha}\nonumber\\&\leq
\left(1-|a_0|^2\right)\frac{(n{\bf r})^N}{1-n{\bf r}}
\end{align}
for $n{\bf r}\in[0,1)$, where $r=(r_1,r_2,\ldots,r_n)$, ${\bf r}=||r||_{\infty}$ and $t=[(N-1)/2]$, $[x]$ denotes the largest integer no more than $x$, where $x$ is a real number.
\end{lem}
\begin{proof} Let $f$ be a holomorphic function in the polydisk $\mathbb{P}\Delta(0;1_n)$ such that $|f(z)|\leq 1$ and
	$f(z)=\sum_{|\alpha|=0}^{\infty}a_{\alpha}z^{\alpha}$ for all $z\in \mathbb{P}\Delta(0;1_n)$. We divide the proof into two cases.\par

\medskip
{\bf Case 1.} Let $N=2m$, where $m\in\mathbb{N}$. Then by Lemma \ref{Lem4}, we have
\begin{align}
\label{sgn1} \sum\limits_{k=2m}\sum\limits_{|\alpha|=k}|a_{\alpha}|r^{\alpha}&=\sum\limits_{k=m}\sum\limits_{|\alpha|=2k}|a_{\alpha}|r^{\alpha}+\sum\limits_{k=m}\sum\limits_{|\alpha|=2k+1}|a_{\alpha}|r^{\alpha}
\\&\leq
\sum\limits_{k=m}\left(1-\sum_{i=0}^{k-1} \sum\limits_{|\alpha|=i}|a_{\alpha}|^2-\frac{\sum\limits_{|\alpha|=k}|a_{\alpha}|^2}{1+|a_0|}\right)(n{\bf r})^{2k}\nonumber\\&\quad\nonumber+\sum\limits_{k=m}\left(1-\sum_{i=0}^k \sum\limits_{|\alpha|=i}|a_{\alpha}|^2\right)(n{\bf r})^{2k+1}\nonumber\\&\leq
\sum_{k=2m}^{\infty} (n{\bf r})^k - 
\left(\sum_{i=0}^{m-1} \sum\limits_{|\alpha|=i}|a_{\alpha}|^2 \right)\left(\sum_{k=2m}^{\infty} (n{\bf r})^k\right)
\nonumber\\&-\left(\frac{1}{1+|a_0|}+ \sum_{k=1}^{\infty} (n{\bf r})^k \right)\sum_{k=m}^{\infty} \sum\limits_{|\alpha|=k}|a_{\alpha}|^2 (n{\bf r})^{2k}\nonumber\\&=
\left(1-\sum_{i=0}^{m-1} \sum\limits_{|\alpha|=i}|a_{\alpha}|^2\right)\frac{(n{\bf r})^{2m}}{1-n{\bf r}}\nonumber\\
&-\left(\frac{1}{1+|a_0|}+ \frac{n{\bf r}}{1-n{\bf r}} \right)\sum_{k=m}^{\infty} \sum\limits_{|\alpha|=k}|a_{\alpha}|^2 (n{\bf r})^{2k}.\nonumber
\end{align}
Using (\ref{sgn1}), a simple computation shows that
\begin{align}
\label{sgn2} &\sum_{k=N}^{\infty}\sum\limits_{|\alpha|=k}|a_{\alpha}|\, r^{\alpha} + \operatorname{sgn}(t)\sum\limits_{k=1}^t\sum\limits_{|\alpha|=k}|a_{\alpha}|^2 \frac{{\bf r}^{N}}{1-{\bf r}}+\left( \frac{1}{1+|a_0|}+\frac{{\bf r}}{1-{\bf r}}\right)
\sum_{k=t+1}^{\infty} \sum\limits_{|\alpha|=k}|a_{\alpha}|^2 {\bf r}^{2k}\nonumber\\&\leq
\sum_{k=2m}^{\infty}\sum\limits_{|\alpha|=k}|a_{\alpha}|\, r^{\alpha} + \operatorname{sgn}(m-1)\sum\limits_{k=1}^{m-1}\sum\limits_{|\alpha|=k}|a_{\alpha}|^2 \frac{(n{\bf r})^{2m}}{1-n{\bf r}}\nonumber\\&+\left( \frac{1}{1+|a_0|}+\frac{n{\bf r}}{1-n{\bf r}}\right)
\sum_{k=m}^{\infty} \sum\limits_{|\alpha|=k}|a_{\alpha}|^2 {\bf r}^{2k}\nonumber\\&\leq
\left(1-\sum_{i=0}^{m-1} \sum\limits_{|\alpha|=i}|a_{\alpha}|^2\right)\frac{(n{\bf r})^{2m}}{1-n{\bf r}}+ \operatorname{sgn}(m-1)\sum\limits_{k=1}^{m-1}\sum\limits_{|\alpha|=k}|a_{\alpha}|^2 \frac{(n{\bf r})^{2m}}{1-n{\bf r}}.
\end{align}

Note that $\operatorname{sgn}(m-1)=0$ if $m=1$ and $=1$ if $m>1$. Therefore the desired result follows easily from (\ref{sgn2}).\par

{\bf Case 2.} $N=2m+1$, where $m\in\mathbb{N}\cup\{0\}$. Then using Lemma \ref{Lem4} and (\ref{sgn1}), we have
\begin{align}
\label{sgn3} &\sum\limits_{k=2m+1}\sum\limits_{|\alpha|=k}|a_{\alpha}|r^{\alpha}\\&=\sum\limits_{k=2(m+1)}\sum\limits_{|\alpha|=k}|a_{\alpha}|r^{\alpha}+\sum\limits_{|\alpha|=2m+1}|a_{\alpha}|r^{\alpha}\nonumber\\&\leq
\left(1-\sum_{i=0}^{m} \sum\limits_{|\alpha|=i}|a_{\alpha}|^2\right)\frac{(n{\bf r})^{2(m+1)}}{1-n{\bf r}}
-\left(\frac{1}{1+|a_0|}+ \frac{n{\bf r}}{1-n{\bf r}} \right)\sum_{k=m+1}^{\infty} \sum\limits_{|\alpha|=k}|a_{\alpha}|^2 (n{\bf r})^{2k}\nonumber\\&\quad+\left(1-\sum\limits_{i=0}^m\sum\limits_{|\alpha|=i}|a_{\alpha}|^2\right)(n{\bf r})^{2m+1}\nonumber\\&=
\left(1-\sum_{i=0}^{m} \sum\limits_{|\alpha|=i}|a_{\alpha}|^2\right)\frac{(n{\bf r})^{2m+1}}{1-n{\bf r}}
-\left(\frac{1}{1+|a_0|}+ \frac{n{\bf r}}{1-n{\bf r}} \right)\sum_{k=m+1}^{\infty} \sum\limits_{|\alpha|=k}|a_{\alpha}|^2 (n{\bf r})^{2k}.\nonumber
\end{align}

Using (\ref{sgn3}), we deduce that
\begin{align}
\label{sgn4} &\sum_{k=N}^{\infty}\sum\limits_{|\alpha|=k}|a_{\alpha}|\, r^{\alpha} + \operatorname{sgn}(t)\sum\limits_{k=1}^t\sum\limits_{|\alpha|=k}|a_{\alpha}|^2 \frac{{\bf r}^{N}}{1-{\bf r}}+\left( \frac{1}{1+|a_0|}+\frac{{\bf r}}{1-{\bf r}}\right)
\sum_{k=t+1}^{\infty} \sum\limits_{|\alpha|=k}|a_{\alpha}|^2 {\bf r}^{2k}\nonumber\\&\leq
\sum_{k=2m+1}^{\infty}\sum\limits_{|\alpha|=k}|a_{\alpha}|\, r^{\alpha} + \operatorname{sgn}(m)\sum\limits_{k=1}^{m}\sum\limits_{|\alpha|=k}|a_{\alpha}|^2 \frac{(n{\bf r})^{2m+1}}{1-n{\bf r}}\nonumber\\&\quad+\left( \frac{1}{1+|a_0|}+\frac{n{\bf r}}{1-n{\bf r}}\right)
\sum_{k=m+1}^{\infty} \sum\limits_{|\alpha|=k}|a_{\alpha}|^2 {\bf r}^{2k}\nonumber\\&\leq
\left(1-\sum_{i=0}^{m} \sum\limits_{|\alpha|=i}|a_{\alpha}|^2\right)\frac{(n{\bf r})^{2m+1}}{1-n{\bf r}}+\operatorname{sgn}(m)\sum\limits_{k=1}^{m}\sum\limits_{|\alpha|=k}|a_{\alpha}|^2 \frac{(n{\bf r})^{2m+1}}{1-n{\bf r}}.
\end{align}

Note that $\operatorname{sgn}(m)=0$ if $m=0$ and $=1$ if $m\geq 1$. Therefore the desired result follows easily from (\ref{sgn4}).

\medskip
To show the inequality (\ref{Lema5}) is sharp, we consider the function
\begin{align}
\label{AM14}f(z)=\frac{a-(z_1+z_2+\ldots+z_n)}{1-a(z_1+z_2+\ldots+z_n)},
\end{align}
where $a\in [0, 1)$. Clearly $f$ is holomorphic in $\mathbb{P}\Delta(0;1/n)$. Since $|a(z_1+z_2+\ldots+z_n)|<1$,
we have 
\begin{align*}
f(z)=a-(1-a^2)\sideset{}{_{k=1}^{\infty}}{\sum} a^{k-1}(z_1+z_2+\ldots+z_n)^k.
\end{align*}
for all $z\in\mathbb{P}\Delta(0;1/n)$. For the point $z=(-r,-r,\ldots,-r)$, we find that
\begin{align}
\label{AM14a}f(z)=a+(1-a^2)\sideset{}{_{k=1}^{\infty}}{\sum} (-a)^{k-1}(n{\bf r})^k.
\end{align}
If we take $a_0=a$ and $a_k=(1-a^2)(-a)^{k-1}$, then it is easy to verify that 
\begin{enumerate}
\item[(i)] $|a_{2k+1}|= 1-|a_0|^2-|a_1|^2-\ldots-|a_k|^2$,\;\;$k=0,1,\ldots$;
\item[(ii)] $|a_{2k}|= 1-|a_0|^2-|a_1|^2-\ldots-|a_{k-1}|^2-\frac{|a_k|^2}{1+|a_0|}$, \;\;$k=1,2,\ldots$.
\end{enumerate}
It is easy see that
\begin{align*}&\sum_{k=N}^{\infty}\sum\limits_{|\alpha|=k}|a_{\alpha}|\, r^{\alpha} + \operatorname{sgn}(t)\sum\limits_{k=1}^t\sum\limits_{|\alpha|=k}|a_{\alpha}|^2 \frac{{\bf r}^{N}}{1-{\bf r}}+\left( \frac{1}{1+|a_0|}+\frac{{\bf r}}{1-{\bf r}}\right)
\sum_{k=t+1}^{\infty} \sum\limits_{|\alpha|=k}|a_{\alpha}|^2 {\bf r}^{2k}\\&=
\sum_{k=N}^{\infty}|a_{k}|\, (n{\bf r})^{\alpha}+
\operatorname{sgn}(t)\sum\limits_{k=1}^t |a_{k}|^2 \frac{(n{\bf r})^{N}}{1-n{\bf r}}+\left( \frac{1}{1+|a_0|}+\frac{n{\bf r}}{1-n{\bf r}}\right)
\sum_{k=t+1}^{\infty}|a_{k}|^2 (n{\bf r})^{2k}.
\end{align*}

\noindent{\bf Case A.} If $N=2m$, where $m\in\mathbb{N}$, then we have
\begin{align*}
\sum\limits_{k=2m}\sum\limits_{|\alpha|=k}|a_{\alpha}|r^{\alpha}&=\sum\limits_{k=m}\sum\limits_{|\alpha|=2k}|a_{\alpha}|r^{\alpha}+\sum\limits_{k=m}\sum\limits_{|\alpha|=2k+1}|a_{\alpha}|r^{\alpha}\\&=
\sum\limits_{k=m}|a_{2k}|(n{\bf r})^{2k}+\sum\limits_{k=m}|a_{2k+1}|(n{\bf r})^{2k+1}\\&=
\left(1-\sum_{k=0}^{m-1}|a_{k}|^2\right)\frac{(n{\bf r})^{2m}}{1-n{\bf r}}
-\left(\frac{1}{1+|a_0|}+ \frac{n{\bf r}}{1-n{\bf r}} \right)\sum_{k=m}^{\infty}|a_{k}|^2 (n{\bf r})^{2k}.
\end{align*}

\noindent{\bf Case B.} Similarly, if $N=2m+1$, where $m\in\mathbb{N}\cup\{0\}$, then we have
\begin{align*}
\sum\limits_{k=2m+1}\sum\limits_{|\alpha|=k}|a_{\alpha}|r^{\alpha}&=\sum\limits_{k=2(m+1)}\sum\limits_{|\alpha|=k}|a_{\alpha}|r^{\alpha}+\sum\limits_{|\alpha|=2m+1}|a_{\alpha}|r^{\alpha}\\&=
\sum\limits_{k=2(m+1)}|a_{k}|(n{\bf r})^{k}+|a_{2m+1}|(n{\bf r})^{2m+1}\\&=
\left(1-\sum_{k=0}^{m-1}|a_{k}|^2\right)\frac{(n{\bf r})^{2m}}{1-n{\bf r}}
-\left(\frac{1}{1+|a_0|}+ \frac{n{\bf r}}{1-n{\bf r}} \right)\sum_{k=m}^{\infty}|a_{k}|^2 (n{\bf r})^{2k}.
\end{align*}
Consequently, we have
\begin{align}\label{sgn5}&\sum_{k=N}^{\infty}\sum\limits_{|\alpha|=k}|a_{\alpha}|\, r^{\alpha} + \operatorname{sgn}(t)\sum\limits_{k=1}^t\sum\limits_{|\alpha|=k}|a_{\alpha}|^2 \frac{({\bf r}^{N}}{1-{\bf r}}+\left( \frac{1}{1+|a_0|}+\frac{{\bf r}}{1-{\bf r}}\right)
\sum_{k=t+1}^{\infty} \sum\limits_{|\alpha|=k}|a_{\alpha}|^2 {\bf r}^{2k}\nonumber\\&=
\left(1-|a_0|^2\right)\frac{(n{\bf r})^N}{1-n{\bf r}},
\end{align}
which shows that the inequality (\ref{Lema5}) is sharp.
\end{proof}
\section{{\bf Proofs of the main results}}\label{Sec-4}

\begin{proof}[\bf Proof of Theorem \ref{Th-2.1}] 
Let us take $z=(z_1,\ldots,z_n)\in \mathbb{P}\Delta(0;1_n)$ such that ${\bf r}=||z||_{\infty}$.
Now by Lemma \ref{Lem1}, we have
\begin{align}
\label{BS1} |f(z)| \leq \frac{||z||_{\infty}+|a_0|}{1+|a_0|||z||_{\infty}}=\frac{{\bf r}+|a_0|}{1+|a_0|{\bf r}}\le \frac{n{\bf r}+|a_0|}{1+|a_0|n{\bf r}}.
\end{align}
By Lemma \ref{Lem5} and inequality (\ref{BS1}), we deduce that 
\begin{align*} \mathcal{A}_1(z, {\bf r})\leq \frac{n{\bf r}+|a_0|}{1+|a_0|n{\bf r}}+\left(1-|a_0|^2\right)\frac{(n{\bf r})^N}{1-n{\bf r}}
=1+\frac{\phi_{n,N}({\bf r})}{(1 + |a_0|{\bf r}^m)(1-n{\bf r})}\leq 1,
\end{align*}
provided $\phi_{n,N}({\bf r}) \leq 0$, where
\begin{align*} \phi_{n,N}({\bf r})&:=(n{\bf r} + |a_0|)(1 - n{\bf r}) + (1 - |a_0|^{2}) (n{\bf r})^{N} (1 + |a_0| n{\bf r}) - (1 - n{\bf r})(1 + |a_0| n{\bf r})\\&= 
(1 - |a_0|)\left[(1 + |a_0|)(1 + |a_0| n{\bf r}) (n{\bf r})^{N} - (1 - n{\bf r})^{2}\right]\\&\leq
(1 - |a_0|)\left[2(1 + n{\bf r}) (n{\bf r})^{N} - (1 - n{\bf r})^{2}\right],
\end{align*}
since $|a_0| < 1$.\vspace{1.2mm}

We see that $\phi_N(n{\bf r}) \leq 0$, if 
\[\psi_N(n{\bf r}):= 2(1 + n{\bf r}) (n{\bf r})^{N} - (1 - n{\bf r})^{2} \leq 0,\]
which holds for $n{\bf r} \leq R_{n,N}$. The first part of the theorem follows.

\smallskip
To show the sharpness of the number $R_{n,N}$, we consider the holomorphic function $f$ in the polydisk $\mathbb{P}\Delta(0;1/n)$ given by (\ref{AM14}). Obviously,
\begin{align}\label{BS2a} |f(z)|=\frac{a+n{\bf r}}{1+an{\bf r}}.
\end{align}

From (\ref{sgn5}) and (\ref{BS2a}), we derive the following
\begin{align}
\label{BS2}
\mathcal{A}_1(z, {\bf r})=\frac{a+n{\bf r}}{1+an{\bf r}}+\left(1-a^2\right)\frac{(n{\bf r})^N}{1-n{\bf r}}=1+\frac{1-a}{(1+na{\bf r})(1-n{\bf r})}H_{a,N}({n\bf r}),
\end{align}
where
\begin{align*}
	H_{a,N}(n{\bf r}):=(1+a)(1+nar)(n{\bf r})^N-(1-nr)^2<
	2(1 + n{\bf r}) (n{\bf r})^{N} - (1 - n{\bf r})^{2}.
\end{align*}
A simple computation shows that the expression on the right-hand side of inequality \eqref{BS2} is greater than $1$ if and only if $H_{a,N}(n\mathbf{r}) > 0$.\vspace{2mm}

Clearly, $H_{a,N}(n\mathbf{r})$ is a strictly increasing function of $a$ on the interval $[0, 1)$. Note that the expression in \eqref{BS2} is less than or equal to $1$ for all $a \in [0, 1]$ if and only if $n\mathbf{r} \leq R_{n,N}$. Finally, for values of $a$ sufficiently close to $1$, we observe that
\begin{align*}
	\lim\limits_{a\to 1^{-}} H_{a,N}(n{\bf r})&=2(1+nr)(n{\bf r})^N-(1-nr)^2\\&=(1+n{\bf r})(1-n{\bf r})\left(\frac{2(n{\bf r})^N}{1-n{\bf r}}-\frac{1-n{\bf r}}{(1+n{\bf r})}\right)>0,
\end{align*}
which shows that $A_1(z)>1$ for ${\bf r}>R_{n,N}$. This proves that the radius $R_{n,N}$ is best possible.
\vspace{0.2mm}

Next, we prove the second part of the theorem. Applying Lemma \ref{Lem5} and utilizing relation (\ref{BS1}), we obtain 
\begin{align} 
\label{BS3}\mathcal{A}_2(z, {\bf r})\leq \frac{n{\bf r}+|a_0|}{1+|a_0|n{\bf r}}+\left(1-|a_0|^2\right)\frac{(n{\bf r})^N}{1-n{\bf r}}=
1+(1-a^2)G_{a,N}(n{\bf r})\leq 1,
\end{align}
provided $\phi_{n,N}({\bf r}) \leq 0$, where
\begin{align*} G_{a,N}(n{\bf r})=\frac{(n{\bf r})^N}{1-n{\bf r}}-\frac{1-(n{\bf r})^2}{(1+an{\bf r})^2}.\end{align*}
It is easy to see that 
\[\frac{\partial G(a,r)}{\partial a} =\frac{(1-(n{\bf r})^2)n{\bf r}}{(1+an{\bf r})^3}>0.\]
Therefore, $G(a,r)$ is a increasing of $a\in [0,1)$ and so 
\[G(a,r)\leq G(1,r)=\frac{(n{\bf r})^N}{1-n{\bf r}}-\frac{1-n{\bf r}}{1+n{\bf r}}.\]
Hence, $\mathcal{A}_2(z, \bf r)\leq 1$ for $n{\bf r}\leq R'_{n,N}$, where $R'_{n,N}$ is the positive root of the equation $(1 + n{\bf r})(n{\bf r})^{N}-(1-n{\bf r})^{2} = 0$.

\smallskip
To show the sharpness of the number $R_{n,N}$, we consider the holomorphic function $f$ in the polydisk $\mathbb{P}\Delta(0;1/n)$ given by (\ref{AM14}). using (\ref{sgn5}) and (\ref{BS2a}), a simple computation shows that
\begin{align}
\label{BS4}
\mathcal{A}_2(z, {\bf r})=\left(\frac{a+n{\bf r}}{1+an{\bf r}}\right)^2+\left(1-a^2\right)\frac{(n{\bf r})^N}{1-n{\bf r}}=1+\frac{(1-a^2)\tilde H_{a,N}({n\bf r})}{(1+na{\bf r})^2(1-n{\bf r})},
\end{align}
where
\begin{align*}
	\tilde H_{a,N}(n{\bf r})&:=(1+nar)^2(n{\bf r})^N-(1-n^2r^2)(1-nr)\\&<
	(1+nr)^2(n{\bf r})^N-(1-n^2r^2)(1-nr).
\end{align*}
A simple computation shows that the expression on the right of \eqref{BS4} is bigger than $1$ if, and only if, $\tilde H_{a,N}(n{\bf r})>0$.\vspace{2mm}

Clearly, $\tilde H_{a,N}(n{\bf r})$ is a strictly increasing function of $a\in[0, 1)$. Note also that the expression \eqref{BS4} is less than or equal to $1$ for all $a\in [0,1]$, only in the case when $n{\bf r}\leq R'_{n,N}$. Finally, for $a$ sufficiently close to $1$, we see that
\begin{align*}
	\lim\limits_{a\to 1^{-}} \tilde H_{a,N}(n{\bf r})&=(1+nr)^2(n{\bf r})^N-(1-n^2r^2)(1-nr)\\&=(1+n{\bf r})^2(1-n{\bf r})\left(\frac{(n{\bf r})^N}{1-n{\bf r}}-\frac{1-n{\bf r}}{(1+n{\bf r})}\right)>0,
\end{align*}
which shows that $\mathcal{A}_2(z, \bf r)>1$ for ${\bf r}>R'_{n,N}$. This proves that the radius $R'_{n,N}$ is best possible.
\end{proof}

\begin{proof}[\bf Proof of Theorem \ref{Th-2.2}] 
Let us take $z=(z_1,\ldots,z_n)\in \mathbb{P}\Delta(0;1_n)$ such that ${\bf r}=||z||_{\infty}$.
We consider the first part. Using Lemma \ref{Lem5} with $N=1$ and (\ref{BS1}), we have
\begin{align*}
\mathcal{A}_3(z, {\bf r})\le \frac{n{\bf r}+|a_0|}{1+n{\bf r}|a_0|}+\frac{(1-|a_0|^{2})n{\bf r}}{1-n{\bf r}}
= \frac{1-(1-|a_0|)B_3(|a_0|,n{\bf r})}{(1+|a_0|n{\bf r})(1-n{\bf r})},
\end{align*}
where
\begin{align*}
B_3(|a_0|,n{\bf r})=(1-|a_0|-|a_0|^{2})(n{\bf r})^{2}-(3+|a_0|)n{\bf r}+1.
\end{align*}

To prove the first inequality in Theorem \ref{Th-2.2}, it suffices to prove that $B_3(|a_0|,n{\bf r})\ge 0$ for all $|a_0|\in[0,1)$ and $n{\bf r}\le r_{a_0}$.\vspace{1.2mm}

 If $|a_0|=\frac{\sqrt{5}-1}{2}$, then $1-|a_0|-|a_0|^{2}=0$ and thus,
$B_1(|a_0|,n{\bf r})\ge 0$ is equivalent to
\begin{align*}
n{\bf r}\le \frac{2}{3+|a_0|}=\frac{2}{5+\sqrt{5}}
=\frac{2}{3+|a_0|+\sqrt{5}(1+|a_0|)}.
\end{align*}

For $|a_0|\in[0,1)\backslash \{\frac{\sqrt{5}-1}{2}\}$, we have
$1-|a_0|-|a_0|^{2}\neq 0$ and thus, we see that
\begin{align*}
\frac{B_1(|a_0|,n{\bf r})}{1-|a_0|-|a_0|^{2}}
=\left(n{\bf r}-\frac{2}{3+|a_0|+\sqrt{5}(1+|a_0|)}\right)
\left(n{\bf r}-\frac{2}{3+|a_0|-\sqrt{5}(1+|a_0|)}\right).
\end{align*}
Then the desired conclusion follows by a simple analysis on the two cases
\begin{align*}
	0\le |a_0|<\frac{\sqrt{5}-1}{2}\;\; \mbox{and}\;\; \frac{\sqrt{5}-1}{2}<|a_0|<1.
\end{align*} 
To prove the sharpness, we let $a\in [0,1)$ and consider the function
\begin{align}\label{NS1} f(z)=\frac{a+(z_1+z_2+\ldots+z_n)}{1+a(z_1+z_2+\ldots+z_n)}.\end{align}
Clearly, $f$ is holomorphic in the polydisk $\mathbb{P}\Delta(0;1/n)$ and
\begin{align*} f(z)=a+(1-a^2)\sum\limits_{k=1}^{\infty} (-a)^{k-1}(z_1+z_2+\ldots+z_n)^k \end{align*}
for all $z\in\mathbb{P}\Delta(0;1/n)$. For the point $z=(r,r,\ldots,r)$,
we find that
\begin{align*} |f(z)|=\frac{a+nr}{1+nar}\end{align*}
and
\begin{align*} f(z)=a+(1-a^2)\sum\limits_{k=1}^{\infty} (-a)^{k-1}(n{\bf r})^k.\end{align*}

For this function, we find that
\begin{align*}
\mathcal{A}_3(z, \bf r)&=\frac{n{\bf r}+a}{1+an{\bf r}}+\frac{(1-a^{2})n{\bf r}}{1-an{\bf r}}
+\left(\frac{1}{1+a}+\frac{n{\bf r}}{1-n{\bf r}}\right)\frac{(1-a^{2})^{2}(n{\bf r})^{2}}{1-a^{2}(n{\bf r})^{2}}\\&=
1-\frac{(1-a)\tilde B_3(a,n{\bf r})}{(1+an{\bf r})(1-n{\bf r})},
\end{align*}
where
\begin{align*}
\tilde B_3(a,n{\bf r}):=(1-a-a^{2})(n{\bf r})^{2}-(3+a)n{\bf r}+1.
\end{align*}
The above inequality is greater than $1$ if and only if $\tilde B_3(a,n{\bf r})<0$.
It is easy to verify that $\tilde B_3(a,n{\bf r})<0$ if and only if
$$n{\bf r}>r_{a}=\frac{2}{3+a+\sqrt{5}(1+a)},$$ which proves the sharpness.
\vspace{1.2mm}

We now proceed to the second part of the proof. Once again, by applying Lemma \ref{Lem5} and relation (\ref{BS1}), we find that
\begin{align*}
\mathcal{A}_4(z, {\bf r})\le \left(\frac{n{\bf r}+|a_0|}{1+n{\bf r}|a_0|}\right)^{2}+\frac{(1-|a_0|^{2})n{\bf r}}{1-n{\bf r}}
= \frac{1-(1-|a_0|^{2})B_4(|a_0|,n{\bf r})}{(1+|a_0|n{\bf r})^{2}(1-n{\bf r})},
\end{align*}
where
\begin{align*}
B_4(|a_0|,n{\bf r}):=(1-|a_0|^{2})(n{\bf r})^{3}-(1+2|a_0|)(n{\bf r})^{2}-2n{\bf r}+1.
\end{align*}
To prove the second inequality in Theorem \ref{Th-2.2}, it suffices to show that the following condition holds
$B_4(|a_0|,n{\bf r})\ge 0$ only for $0\le n{\bf r}\le r_{a_0}^{\ast}$. Elementary calculations show that the following holds
\begin{align*}
B_4\!\left(|a_0|,\frac{1}{3n}\right)=\frac{1}{27}(1-|a_0|)(7+|a_0|)>0,
\end{align*}
\begin{align*}
B_4\!\left(|a_0|,\frac{1}{n(2+|a_0|)}\right)=
-\frac{(1-|a_0|)(1+|a_0|)^{2}}{(2+|a_0|)^{3}}<0,
\end{align*}
and
\begin{align*}
\frac{\partial B_4(|a_0|,n{\bf r})}{\partial (n{\bf r})}
=-3|a_0|^{2}(n{\bf r})^{2}-4|a_0|n{\bf r}-(1-n{\bf r})(1+3n{\bf r})-1<0.
\end{align*}
Thus, the desired conclusion is immediately verified. The optimality (or sharpness) of the constant $r_{a_0}^{\ast}$ is shown by an analogous computation, which we omit for brevity.
\end{proof}
\begin{proof}[\bf Proof of Theorem \ref{Th-2.3}]
Let us take $z=(z_1,\ldots,z_n)\in \mathbb{P}\Delta(0;1_n)$ such that ${\bf r}=||z||_{\infty}$. By simple calculation, we know that the following inequality holds
\begin{align*}
\frac{2n{\bf r}}{1-(n{\bf r})^{2}}\le 1 \quad \text{if } 0 \le n{\bf r} \le \sqrt{2}-1.
\end{align*}
By Lemma \ref{Lem2}, we have the following inequality
\begin{align}\label{SB1} |Df(z)|\leq \frac{1-|f(z)|^2}{1-{\bf r}^2} n{\bf r}\leq \frac{1-|f(z)|^2}{1-(n{\bf r})^2} n{\bf r}.
\end{align}
Suppose that
\begin{align*} \Phi(X)=X+\lambda(1-X^{2}),
\end{align*}
where $X=|f(z)|$ and $\lambda = \frac{n{\bf r}}{1-(n{\bf r})^2}$. It is easy to see that $\Phi(X)\le \Phi(X_{0})$ when 
\begin{align*}
X\le X_{0}=\frac{n{\bf r}+|a_{0}|}{1+n{\bf r}|a_{0}|}.
\end{align*}
Using Lemma \ref{Lem5} with $N=2$, along with relations (\ref{BS1}) and (\ref{SB1}), we obtain
\small{\begin{align*}
\mathcal{I}(z, {\bf r})&\le |f(z)|+\frac{n{\bf r}}{1-(n{\bf r})^{2}}\,(1-|f(z)|^{2})+\frac{(1-|a_{0}|^{2})(n{\bf r})^{2}}{1-n{\bf r}}\\&\le  
\frac{n{\bf r} +|a_{0}|}{1+n{\bf r}|a_{0}|} + \frac{n{\bf r}}{1-(n{\bf r})^{2}}
\left( 1 - \left( \frac{n{\bf r}+|a_{0}|}{1+n{\bf r}|a_{0}|}\right)^{2} \right)
+ \frac{(1-|a_{0}|^{2})(n{\bf r})^{2}}{1-n{\bf r}}\\&=
\frac{n{\bf r} +|a_{0}|}{1+n{\bf r}|a_{0}|}
+ \frac{n{\bf r}(1-|a_{0}|^{2})}{(1+n{\bf r}|a_{0}|)^{2}}
+ \frac{(1-|a_{0}|^{2})(n{\bf r})^{2}}{1-n{\bf r}}\\&
= 1 + \frac{1-|a_{0}|}{(1+|a_{0}|n{\bf r})^{2}(1-n{\bf r})}\\&\quad \times
\big[-1 + 3n{\bf r} - (n{\bf r})^{2} + (2(n{\bf r})^{2}+(n{\bf r})^{3})|a_{0}| + (2(n{\bf r})^{3}+(n{\bf r})^{4})|a_{0}|^{2} + (n{\bf r})^{4}|a_{0}|^{3}\big]\\&
\le 1 + \frac{1-|a_{0}|}{(1+|a_{0}|r)^{2}(1-r)}\\&\quad \times
[-1+3n{\bf r}-(n{\bf r})^{2}+(2(n{\bf r})^{2}+(n{\bf r})^{3})+(2(n{\bf r})^{3}+(n{\bf r})^{4})+(n{\bf r})^{4}]\\&
= 1 + \frac{2(1-|a_{0}|)(1+(n{\bf r})^{2})}{(1+|a_{0}|n{\bf r})^{2}(1-n{\bf r})}
\left(n{\bf r}-\frac{\sqrt{17}-3}{4}\right)
\left(n{\bf r}+\frac{\sqrt{17}+3}{4}\right)\\&\le 1,
\end{align*}}
for 
\begin{align*}
	0 \le n{\bf r} \le \frac{\sqrt{17}-3}{4} < \sqrt{2}-1.
\end{align*}
\vspace{0.2mm}
It is easy to see that
\begin{align*}
	\frac{1-n{\bf r}}{1-(n{\bf r})^{2}}\ge 0\;\; \mbox{if}\;\;0 \le n{\bf r} \le \frac{1}{2}
\end{align*} 
Again, following the approach of the previous case, we find that
\begin{align*}
\mathcal{J}(z, \bf r)&\le \left( \frac{n{\bf r}+|a_{0}|}{1+n{\bf r}|a_{0}|}\right)^{2}
+ \frac{n{\bf r}(1-|a_{0}|^{2})}{(1+n{\bf r}|a_{0}|)^{2}+ (1-|a_{0}|^{2})(n{\bf r})^{2}}\left(\frac{1}{1-n{\bf r}}\right)\\&=
1 + \frac{\left(1-|a_{0}|^{2}\right)\left(-(1-n{\bf r})(1-n{\bf r}-(n{\bf r})^{2}) + (n{\bf r})^{2}(1+n{\bf r}|a_{0}|)^{2}\right)}{(1+|a_{0}|n{\bf r})^{2}(1-n{\bf r})}
\\&
\le 1 + \frac{\left(1-|a_{0}|^{2}\right)\left(-1+2n{\bf r}+(n{\bf r})^{2}+(n{\bf r})^{3}+(n{\bf r})^{4}\right)}{(1+|a_{0}|n{\bf r})^{2}(1-n{\bf r})}
\\&
\le 1,
\end{align*}
for $0 \le {\bf r} \le r_{0}<\frac{1}{2n}$.\vspace{1.2mm}

To show the sharpness, we consider the holomorphic function $f$ in the polydisk $\mathbb{P}\Delta(0;1/n)$ given by (\ref{NS1}).
Now for $z=(r,r,\ldots,r)$, we get
\small{\begin{align*}
\mathcal{J}(z, \bf r)&=\left(\frac{n{\bf r}+a}{1+n{\bf r}a}\right)^{2}
+ \frac{(1-a^{2})n{\bf r}}{(1+an{\bf r})^{2}}+\frac{(1-a^{2})a(n{\bf r})^{2}}{1-an{\bf r}}
+ \frac{1+an{\bf r}}{(1+a)(1-n{\bf r})}\left(\frac{(1-a^{2})^{2}(n{\bf r})^{2}}{1-a^{2}(n{\bf r})^2}\right)\\&=
1 + \frac{\left(1-a^{2}\right)\left(-1 + 2n{\bf r} + (n{\bf r})^{2} - (n{\bf r})^{3} + 2(n{\bf r})^{3}a + (n{\bf r})^{4}a^{2}\right)}{(1+an{\bf r})^{2}(1-n{\bf r})}.
\end{align*}}

The last expression is larger than $1$ if, and only if,
\begin{align*}
\tilde J(a,n{\bf r})= -1+2n{\bf r}+(n{\bf r})^{2}-(n{\bf r})^{3}+2(n{\bf r})^{3}a+(n{\bf r})^{4}a^{2}>0
\end{align*}
for all $a \in [0,1)$ and all $n{\bf r}$ in some subset of $[0,1)$.
The equivalent condition implies $\tilde J(a,n{\bf r})\ge 0$ by $a\to 1$. It is easy to see that $\tilde J(1,n{\bf r})\ge 0$ if and only if ${\bf r}\ge r_{0}$. This shows the sharpness.
\end{proof}

\begin{proof}[\bf Proof of Theorem \ref{Th-2.4}]
Let us take $z=(z_1,\ldots,z_n)\in \mathbb{P}\Delta(0;1_n)$ such that ${\bf r}=||z||_{\infty}$. 
By simple calculations, we can see that
\begin{align*}
n{\bf r} \le n \tilde R_{n,N} < \frac{1}{2} \quad \text{if and only if} \quad \frac{2(n {\bf r})^N}{(1+n{\bf r})(1-2{\bf r})(1-{\bf r})^{N-1}} \le 1.
\end{align*}

Using Lemmas \ref{Lem1} and \ref{Lem2}, we have
\begin{align*}
\mathcal{M}(z, \bf r):=&|f(z)| +\sum_{k=N}^{\infty} \ \sum_{|\alpha|=k}\frac{1}{\alpha!}\left|\frac{\partial^{|\alpha|} f(z)}{\partial z_1^{\alpha_1}\cdots \partial z_n^{\alpha_n}}\right| |z|^{\alpha}\\
\le & |f(z)|+ \sum_{k=N}^{\infty} \ \sum_{|\alpha|=k}\frac{1-|f(z)|^2}{(1-||z||_{\infty}^2)^{|\alpha|}}(1+||z||_{\infty})^{|\alpha|-1}|z|^\alpha\\
\le  &|f(z)| + (1-|f(z)|^2)\sum_{k=N}^{\infty} \frac{(1+(n{\bf r}))^{k-1} (n{\bf r})^k}{(1-(n{\bf r})^2)^k}\\
=& |f(z)| + (1-|f(z)|^2) \frac{(n{\bf r})^N}{(1+n{\bf r})(1-2n{\bf r})(1-n{\bf r})^{N-1}}.
\end{align*}
Applying the same argument as in the proof of Theorem \ref{Th-2.3}, we proceed as follows to obtain
\small{\begin{align*}
\mathcal{M}(z, {\bf r})\leq &\frac{|a_0|+n{\bf r}}{1+n|a_0|{\bf r}} +\left(1-\left(\frac{|a_0|+n{\bf r}}{1+n|a_0|{\bf r}}\right)^2\right) \frac{(n{\bf r})^N}{(1+n{\bf r})(1-2n{\bf r})(1-n{\bf r})^{N-1}}\nonumber\\
\le & \frac{(|a_0|+n{\bf r})(1+n|a_0|{\bf r})(1-n{\bf r})^N(1-2{\bf r}) + (1-|a_0|^2)(1-n{\bf r})^2 (n{\bf r})^N}{(1-n{\bf r})^N(1-2n{\bf r})(1+n|a_0|{\bf r})^2}:=\omega_{n,N}(n{\bf r})
\end{align*}}
$0 \le n{\bf r} \le n\tilde R_{n,N} < \frac{1}{2}$.\vspace{1.2mm}

We see that $\omega_{n,N}(n{\bf r}) \le 1$ if $\nu_{nn,N}(n{\bf r}) \le 0$, where
\begin{align*}
\nu_{n,N}(n{\bf r})&= (|a_0|+n{\bf r})(1+n|a_0|{\bf r})(1-n{\bf r})^N(1-2n{\bf r}) + (1-|a_0|^2)(1-n{\bf r})^2(n{\bf r})^N\\&
\quad- (1-n{\bf r})^N(1-2n{\bf r})(1+n|a_0|{\bf r})^2\\&=
(1-|a_{0}|)\left(-1 + (3-|a_{0}|)n{\bf r})+(3|a_{0}|-2)(n{\bf r})^{2}- 2|a_{0}|(n{\bf r})^{3}\right)(1-n{\bf r})^{N}\\&\quad+
 (1-|a_{0}|)(1+|a_{0}|)\, (n{\bf r})^{N} (1-n{\bf r})^{2}\\&=
 (1 - |a_{0}|)\left((-1 + 3n{\bf r} - 2(n{\bf r})^{2})(1-n{\bf r})^{N}+(n{\bf r})^{N}(1-n{\bf r})^{2}\right)\\&=
 (1 - |a_{0}|)|a_{0}|n{\bf r}(1-n{\bf r})^{2}\left((n{\bf r})^{N-1}-(1-2n{\bf r})(1-n{\bf r})^{N-1}\right).
\end{align*}
To complete the proof, we now consider following two cases.\vspace{1.2mm}

\textbf{Case 1.} Let $n{\bf r} \le \tilde R_{n,N,1}$, where $\tilde R_{n,N,1}$ is the minimum positive root of the equation 
\begin{align*}
\phi_{n,N}(n{\bf r}) = (1-2n{\bf r})(1-n{\bf r})^{N-1} - (n{\bf r})^{N-1} = 0.
\end{align*}
Since $(n{\bf r})^{N-1} - (1-2n{\bf r})(1-n{\bf r})^{\,N-1} \le 0$ and $|a_0|<1$, we have
\begin{align*}
\nu_{n,N}(n{\bf r}) \le &(1-|a_0|)(1-n{\bf r})^2 \left(\frac{1}{n{\bf r}}\left( (n{\bf r})^{N-1} - (1-2n{\bf r})(1-n{\bf r})^{N-1}\right) \right)\\
\le &(1-|a_0|)(1-n{\bf r})^2\left( (n{\bf r})^{N-1} - (1-2n{\bf r})(1-n{\bf r})^{N-1}\right) \le 0.
\end{align*}

\textbf{Case 2.} Let $\tilde R_{n,N,1} < n{\bf r} \le \tilde R_{n,N}$. Since $\tilde R_{n,N,1}<\tilde R_{n,N}$ and 
\begin{align*}
(n{\bf r})^{N-1} - (1-2n{\bf r})(1-n{\bf r})^{\,N-1}>0 \quad \text{for } n{\bf r}>\tilde R_{n,N,1},
\end{align*}
we have
\begin{align*}
\nu_{n,N}(n{\bf r}) \le & (1-|a_0|)\left[(-1+3n{\bf r}-2(n{\bf r})^2)(1-n{\bf r})^N + (n{\bf r})^N(1-n{\bf r})^2\right]\\
&+(1-|a_0|)n{\bf r}(1-n{\bf r})^2 \left((n{\bf r})^{\,N-1}-(1-2n{\bf r})(1-n{\bf r})^{N-1}\right)\\
\le &(1-|a_0|)(1-n{\bf r})^2 \left(2(n{\bf r})^N-(1+n{\bf r})(1-2n{\bf r})(1-n{\bf r})^{N-1}\right) \le 0.
\end{align*}
Hence, the desired inequality follows.\vspace{1.2mm}

To show the sharpness of $\tilde R_{n,N}$, let $a \in [0,1)$ and consider the holomorphic function $f$ in $\mathbb{P}\Delta(0;1/n)$ given by (\ref{AM14}). For the point $z=(-r,-r,\ldots,-r)$, we find that

\begin{align}\label{pp1}
	\mathcal{M}(z, {\bf r}):=\frac{a-n{\bf r}}{1-na{\bf r}} + (1-a^2)\frac{a^{N-1} n({\bf r})^N}{(1-na{\bf r})^N (1-2n a{\bf r})}\;\; \mbox{for}\;\; n{\bf r} < \frac{1}{2a}.
\end{align}
This expression is greater than $1$ if, and only, if
\begin{align}
	\label{pp2}
	(1-a)\left((-1+(2a-1)n{\bf r}+2a(n{\bf r})^2)(1-na{\bf r})^{N-1} + (1+a)a^{N-1}(n{\bf r})^N\right) > 0.
\end{align}
We define
\begin{align*}
P_1(a,n{\bf r}) := (-1+(2a-1)n{\bf r}+2a(n{\bf r})^2)(1-na{\bf r})^{N-1} + (1+a)a^{N-1} (n {\bf r})^N.
\end{align*}
We find through computation that $\frac{\partial P_4}{\partial a} \ge 0$ when $n\mathbf{r} < \frac{1}{2}$. It follows immediately that
\begin{align*}
P_1(a,n{\bf r})& \le P_1(1,n{\bf r})\\&= (-1+n{\bf r}+2n({\bf r})^2)(1-n{\bf r})^{N-1} + 2(n{\bf r})^N\\&
= 2(n{\bf r})^N - (1+n{\bf r})(1-2n{\bf r})(1-n{\bf r})^{N-1}
\end{align*}
holds for $n{\bf r}<\frac{1}{2}$. Therefore, equation (\ref{pp1}) is less than or equal to $1$ for all $a \in [0, 1)$ if and only if $n\mathbf{r} \le \tilde R_{n,N}$. Furthermore, letting $a \to 1$ in (\ref{pp2}) shows that equation (\ref{pp1}) is greater than $1$ when $n\mathbf{r} > \tilde R_{n,N}$. This establishes the sharpness of the result.
\vspace{1.2mm}

Next, we prove the second part. Using Lemmas \ref{Lem1} and \ref{Lem2}, we obtain the following
\begin{align*}
\mathcal{N}(z, {\bf r}):&=|f(z)|^2 +\sum_{k=N}^{\infty} \ \sum_{|\alpha|=k}\frac{1}{\alpha!}\left|\frac{\partial^{|\alpha|} f(z)}{\partial z_1^{\alpha_1}\cdots \partial z_n^{\alpha_n}}\right| |z|^{\alpha}\\
&\le |f(z)|^2+ \sum_{k=N}^{\infty} \ \sum_{|\alpha|=k}\frac{1-|f(z)|^2}{(1-||z||_{\infty}^2)^{|\alpha|}}(1+||z||_{\infty})^{|\alpha|-1}|z|^\alpha\\
&\le |f(z)|^2 + (1-|f(z)|^2)\sum_{k=N}^{\infty} \frac{(1+(n{\bf r}))^{k-1} (n{\bf r})^k}{(1-(n{\bf r})^2)^k}\\
&= |f(z)|^2 + (1-|f(z)|^2) \frac{(n{\bf r})^N}{(1+n{\bf r})(1-2n{\bf r})(1-n{\bf r})^{\,N-1}}.
\end{align*}
Applying the same argument as in the proof of Theorem \ref{Th-2.3}, we arrive at the following expression
\begin{align}
\label{BSN1}
\mathcal{N}(z, \bf r)&\leq \left(\frac{|a_0|+n{\bf r}}{1+n|a_0|{\bf r}} \right)^2
+\left(1-\left(\frac{|a_0|+n{\bf r}}{1+n|a_0|{\bf r}}\right)^2\right) 
\frac{(n{\bf r})^N}{(1+n{\bf r})(1-2n{\bf r})(1-n{\bf r})^{\,N-1}}\nonumber\\
&\le \frac{(1-(n{\bf r})^{N})(1-n{\bf r})}{(1+n{\bf r})(1-2n{\bf r})(1-n{\bf r})^{N}} |f(z)|^{2}
+ \frac{(n{\bf r})^{N}(1-n{\bf r})}{(1+n{\bf r})(1-2n{\bf r})(1-n{\bf r})^{N}} \nonumber\\
&\le \frac{(|a_{0}|+n{\bf r})^{2}(1-n{\bf r})^{N}(1-2n{\bf r})
+(1-|a_{0}|^{2})(n{\bf r})^{N}(1-n{\bf r})^{2}}
{(1-n{\bf r})^{N}(1+|a_{0}|n{\bf r})^{2}(1-2n{\bf r})}.
\end{align}
Clearly, the expression in (\ref{BSN1}) is less than or equal to $1$ provided $\omega_{n,N}(r) \le 1$, where $\omega_{n,N}(r)$ is defined as
\begin{align*}
\tilde \omega_{n,N}(r)
:= \frac{ (|a_{0}|+n{\bf r})^{2}(1-n{\bf r})^{N}(1-2n{\bf r}) + (1-|a_{0}|^{2})(n{\bf r})^{N}(1-n{\bf r})^{2} }
{ (1-n{\bf r})^{N}(1+|a_{0}|n{\bf r})^{2}(1-2n{\bf r})}.
\end{align*}
It is easy to see that $\tilde \omega_{n,N}(r)\le 1$ if $\tilde \nu_{n,N}(r)\le 0$, where
\begin{align*}
\tilde \nu_{n,N}(r)&= (|a_{0}|+n{\bf r})^{2}(1-n{\bf r})^{N}(1-2n{\bf r})
  + (1-|a_{0}|^{2})(n{\bf r})^{N}(1-n{\bf r})^{2}\\&
  - (1-n{\bf r})^{N}(1+|a_{0}|n{\bf r})^{2}(1-2n{\bf r}) \\
&= (1-|a_{0}|^{2}) \Big((1-n{\bf r})^{N}(-1+(n{\bf r})^{2}+2n{\bf r}-2(n{\bf r})^{3}) + (n{\bf r})^{N}(1-n{\bf r})^{2}\Big) \\
&= (1-|a_{0}|^{2}) \Big((1-n{\bf r})^{N}(1-n{\bf r})(1+n{\bf r})(2n{\bf r}-1) + (n{\bf r})^{N}(1-n{\bf r})^{2}\Big).
\end{align*}
Clearly, $\tilde \nu_{n,N}(r)\le 0$, provided that
\begin{align*}
(1+n{\bf r})(1-2n{\bf r})(1-n{\bf r})^{N-1} - (n{\bf r})^{N} \ge 0 ,
\end{align*}
which holds for $r\le R_{n,N}'$, where $R_{n,N}'$ is as in the statement of the theorem.

To show the sharpness of $\tilde R_{n,N}'$, let $a \in [0,1)$ and consider the holomorphic function $f$ in $\mathbb{P}\Delta(0;1/n)$ given by (\ref{AM14}). For the point $z=(-r,-r,\ldots,-r)$, we find that
\begin{align}\label{MSN2}
\mathcal{N}(z, {\bf r})&:=\left(\frac{a-n{\bf r}}{1-na{\bf r}}\right)^2 + (1-a^2)\frac{a^{N-1} n({\bf r})^N}{(1-na{\bf r})^N (1-2n a{\bf r})}\nonumber\\
&= \frac{ (a-n{\bf r})^{2}(1-an{\bf r})^{N-2}(1-2an{\bf r}) + (1-a^{2})a^{N-1}(n{\bf r})^{N} }{ (1-an{\bf r})^{N}(1-2an{\bf r}) }. 
\end{align}
Equation (\ref{MSN2}) is larger than $1$ if and only if
\begin{align}
\label{MSN3}
(1-a^{2})\Big((-1+2an{\bf r}+(n{\bf r})^{2}-2a(n{\bf r})^{3})(1-an{\bf r})^{N-2} + a^{N-1}(n{\bf r})^{N}\Big) > 0. 
\end{align}
Elementary calculation, coupled with taking the limit $a \to 1$ in (\ref{MSN3}), shows that the expression in (\ref{MSN2}) exceeds $1$ when $r > R_{n,N}'$.
\end{proof}
\vspace{5mm}
\noindent\textbf{Conflict of interest:} The authors declare that there is no conflict  of interest regarding the publication of this paper.\vspace{1.2mm}

\noindent {\bf Funding:} Not Applicable.\vspace{1.2mm}

\noindent\textbf{Data availability statement:}  Data sharing not applicable to this article as no datasets were generated or analysed during the current study.\vspace{1.2mm}

\noindent {\bf Authors' contributions:} All the authors have equal contributions in preparation of the manuscript.

\end{document}